\documentclass[a4paper,12pt, twoside, reqno]{amsart}
\usepackage{amssymb}
\usepackage{amsmath}
\usepackage[dvips]{graphicx}
\usepackage[colorlinks=true, linkcolor=black, citecolor=black]{hyperref}
\usepackage{pstricks}
\usepackage{verbatim}
\usepackage{float}
\usepackage{placeins}
\usepackage{multicol}
\usepackage{subcaption} %  for subfigures environments 
\usepackage{listings}
\newtheorem{theorem}{Theorem}

\newtheorem{lem}{Lemma}

\newtheorem{defn}{Definition}
\newtheorem{ex}{Example}
\newtheorem{remark}{Remark}[section]

\title{An inertial self-adaptive algorithm for solving split feasibility problems and fixed point problems  in the class of demicontractive mappings}
\author{Vasile Berinde$^{1,2}$}
\begin{document}

\begin{abstract}
We propose a hybrid inertial self-adaptive algorithm for solving the split feasibility problem and fixed point problem in the class of demicontractive mappings. Our results are very general and extend several related results existing in literature from the class of nonexpansive or quasi-nonexpansive mappings to the larger class of demicontractive mappings. Examples to illustrate numerically the effectiveness of the new analytical results  are presented.
\end{abstract}
\maketitle \pagestyle{myheadings} \markboth{Vasile Berinde} {}

\section{Introduction}

Let $H_1$, $H_2$ be real Hilbert spaces, $C$, $Q$ nonempty convex closed subsets of $H_1$ and $H_2$, respectively, and $A : H_1\rightarrow H_2$ a bounded linear operator. The split feasibility problem ($SFP$, for short) is asking to find a point
\begin{equation}\label{sfp}
x\in C \textnormal{ such that } Ax\in Q.
\end{equation}
Under the hypothesis that the $SFP$ is consistent, i.e., \eqref{sfp} has a solution, this is usually denoted by
\begin{equation}\label{1.2}
SFP(C,Q):=\{x\in C \textnormal{ such that } Ax\in Q\},
\end{equation}
 to indicate the two sets involved.
 
 The split feasibility problem  includes many important problems in nonlinear analysis modelling a wide range of inverse problems originating in real world: signal processing, image reconstruction problem of X-ray tomography, statistical learning etc., a fact that challenged researchers to construct robust and efficient iterative algorithms that solve \eqref{sfp}.

Such an algorithm, known under the name of (CQ) algorithmhas been proposed by Byrne \cite{Byrne04}, who constructed it by using the fact thatthe  $SFP$ \eqref{sfp} is equivalent to the following fixed point problem
\begin{equation}\label{FPE}
x = P_C\left((I+\gamma A^*(P_Q-I)A\right)x, x\in C,
\end{equation}
where $P_C$ and $P_Q$ stand for the orthogonal (metric) projections onto the sets $C$ and $Q$, respectively, $I$ is the identity map, $\gamma$ is a positive constant and $A^*$ denotes the adjoint of $A$.

By simply applying the Picard iteration corresponding to the fixed point problem \eqref{FPE}, we get the  $(CQ)$ algorithm,  which is thus generated by an initial value $x_1\in H_1$ and the one step iterative scheme
\begin{equation}\label{1.4}
x_{n+1} = P_C\left((I+\gamma_n A^*(P_Q-I)A\right)x_n,  n\geq 0,
\end{equation}
where the step size $\gamma_n\in \left(0,\frac{2}{\|A\|^2}\right)$.% and $\lambda$ is the spectral radius of the operator $A^*A$.

\noindent If, for example, one considers the function 
\begin{equation}\label{norma}
f(x)=\frac{1}{2} \|(I-P_Q) Ax\|^2,
\end{equation}
then we have
\begin{equation}\label{nabla}
\nabla f(x)=A^* (I-P_Q) Ax,
\end{equation}
which indicates the fact that \eqref{1.4} is a particular gradient projection type algorithm. 
Of course, this is valid in a more general case: if we have a Fr\' echet differentiable real-valued valued function $f:C\rightarrow \mathbb{R}$ and we search for a minimizer of the problem
\begin{equation}\label{minim}
\textnormal{find }\min\limits_{x\in C} f(x), 
\end{equation}
then by means of an equivalent fixed point formulation, i.e., 
\begin{equation}\label{fpp}
x=P_C\left(x-\gamma \nabla f(x)\right) 
\end{equation}
one obtains the gradient-projection algorithm
\begin{equation}\label{gpa}
x_{n+1}=P_C\left(x_n-\gamma \nabla f(x_n)\right), n\geq 0, 
\end{equation}
which coincides with \eqref{1.4} in the particular case of $f$ given by \eqref{norma}, see \cite{Xu11} for more details.

It is known that when the iteration mapping $$P_C\left((I+\gamma A^*(P_Q-I)A\right)$$ involved in the $(CQ)$ algorithm \eqref{1.4} is of nonexpansive type, then the $(CQ)$ algorithm converges strongly to a fixed point of it, that is, to a solution of the $SFP$ \eqref{sfp} (see \cite{Byrne04}, for more details). 

But in applications, there are at least two major difficulties in implementing the algorithm \eqref{1.4}:
\begin{enumerate}
\item the selection of the step size depends on the operator norm, and its computation is not an easy task at all; 
\item the implementation of the projections  $P_C$ and $P_Q$, depending on the geometry of the two sets $C$ and $Q$, could be very difficult or even impossible.
\end{enumerate}

In order to overcome the above mentioned computational difficulties in a gradient-projection type algorithm, researchers proposed some ways to avoid the calculation of $\|A\|$. Another way to surpass the computation of the norm of $A$ has been suggested by Lopez et al. \cite{Lopez}, who proposed the following formula for expressing the step size sequence $\gamma_n$:
\begin{equation}\label{1.5}
\gamma_n:=\frac{\rho_n f(x_n)}{\|\nabla f(x_n)\|^2},\,n\geq 1,
\end{equation}
where $\rho_n$ is a sequence of positive real numbers in the interval $(0,4)$.

Another fixed point approach for solving the $(SFP)$ \eqref{sfp} in the class of nonexpansive mappings is due to Qin et al. \cite{Qin}, who considered a viscosity type algorithm given by
\begin{equation}\label{1.6}
\begin{cases}
x_1\in C \textnormal{ arbitrary }\\
y_n=P_C \left((1-\delta_n) x_n-\tau_n A^*(I-P_Q)A x_n\right)+\delta_n S x_n,\\
x_{n+1}=\alpha_n g(x_n)+\beta_n x_n+\gamma_n y_n,\,n\geq 1,
\end{cases}
\end{equation}
where $g:C\rightarrow C$ is a Banach contraction, $T:C\rightarrow C$ is a nonexpansive mapping with $Fix\,(T)\neq \emptyset$, $\{\alpha_n\}$, $\{\alpha_n\}$, $\{\beta_n\}$, $\{\gamma_n\}$,  $\{\delta_n\}$ and $\tau_n$ are sequences in $(0,1)$ that satisfy some appropriate conditions, denoted by $(C_1)$-$(C_5)$. 

Under these assumptions,  Qin et al. \cite{Qin} proved that the sequence $\{x_n\}$ generated by the algorithm \eqref{1.6} converges strongly to some $x^*\in Fix\,(T)\cap SFP(C,Q)$ and $x^*$ is the unique solution of the variational inequality
\begin{equation}\label{1.7}
\langle x-x^*,g(x^*)-x^*\rangle \leq 0,\,\forall x\in Fix\,(T)\cap SFP(C,Q).
\end{equation}
Subsequently, Kraikaew et al. \cite{Kraikaew} have weakened the assumptions $(C_1)$, $(C_2)$ and $(C_4)$ in Lopez et al. \cite{Lopez} and obtained the same convergence result by a slightly simplified proof. 

More recently, Wang et al. \cite{Wang-Xu} extended the previous results in three ways:
\begin{enumerate}
\item by weakening the conditions on the parameters $\{\alpha_n\}$, $\{\alpha_n\}$, $\{\beta_n\}$. $\{\gamma_n\}$ and $\{\delta_n\}$ involved in the algorithm \eqref{1.6};
\item by inserting an inertial term in the algorithm \eqref{1.6} in such a way that for choosing the step size it is no more need to calculate the norm of the operator $A$; 
\item by considering the larger class of quasi-nonexpansive mappings instead of nonexpansive mappings (which were considered in the previous papers).
\end{enumerate}

\noindent Starting from the developments presented before, the following question naturally arises:
\medskip

{\bf Question.} Is it possible to extend the results in Wang et al. \cite{Wang-Xu} to more general classes of mappings that strictly include the class of quasi-nonexpansive mappings ?
\medskip

The aim of this paper is to answer this question in the affirmative, see Theorem \ref{th1} below and also its supporting illustration (Example \ref{ex1}). We actually show that we can establish a strong convergence theorem for Algorithm 1, which is obtained from the inertial algorithm  \eqref{1.6}  used in \cite{Wang-Xu}  by  inserting an averaged component. We are thus able to show that one can solve the split feasibility problem and the fixed point problem in the class of demicontractive mappings, too. 

%Our new algorithm is designated for solving the $SFP$ in the class of demicontractive mappings. 

Our main result (Theorem \ref{th1}) shows that the new algorithm  converges strongly to an element $x^*\in Fix\,(T)\cap SFP(C,Q)$ which uniquely solves the variational inequality \eqref{1.7}. 

By doing this, we improve significantly the previous related results in literature since, by considering averaged mappings in gradient projection type algorithms, one get important benefits, see the motivation in the excellent paper by Xu \cite{Xu11}.

\section{Preliminaries}

Throughout this section, $H$ denotes a real Hilbert space with norm and inner product denoted as usually by $\|\cdot\|$ and $\langle \cdot,\cdot\rangle$, respectively. Let $C\subset H$ be a closed and convex set and $T:C\rightarrow C$ be a self mapping. Denote by
$$
Fix\,(T)=\{x\in C: Tx=x\}
$$
the set of fixed points of  $T$. In the present paper we consider the classes of nonexpansive type mappings introduced by the next definition.

\begin{defn}\label{def1}
The mapping $T$ is said to be:
\smallskip

1) {\it nonexpansive} if 
\begin{equation}\label{ne}
\|Tx-Ty\|\leq \|x-y\|,\,\textnormal{ for all } x,y\in C.
\end{equation}

2) {\it quasi-nonexpansive} if $Fix\,(T)\neq \emptyset$ and 
\begin{equation}\label{qne}
\|Tx-y\|\leq \|x-y\|,\, \textnormal{ for all } x\in C \textnormal{ and } y\in Fix\,(T).
\end{equation}

3) {\it $k$-strictly pseudocontractive} of the Browder-Petryshyn type if there exists $k<1$ such that
\begin{equation}\label{strict}
\|Tx-Ty\|^2\leq \|x-y\|^2+k\|x-y-Tx+Ty\|^2, \forall x,y\in C.
\end{equation}

4) {\it $k$-demicontractive}  or {\it quasi $k$-strictly pseudocontractive}  (see \cite{BPR23}) if $Fix\,(T)\neq \emptyset$ and there exists  a positive number $k<1$ such that 
\begin{equation}\label{demi}
\|Tx-y\|^2\leq \|x-y\|^2+k\|x-Tx\|^2,
\end{equation}
for all $x\in C$ and $y\in Fix\,(T)$. %We also say that $T$ is $k$-demicontractive.
\end{defn}
\medskip

For the scope of this paper, it is important to note that any quasi-nonexpansive mapping is demicontractive but the reverse is no more true, as shown by the following example. 

\begin{ex} [\cite{Ber24}, Example 2.5]\label{ex1}
Let $H$ be the real line with the usual norm and $C = [0, 1]$. Define $T$ on $C$ by $T x = \dfrac{7}{8}$, if $0\leq x<1$ and $T 1 = \dfrac{1}{4}$. Then: 
1) $Fix\,(T)\neq \emptyset$; 2) $T$ is demicontractive; 3) $T$ is not nonexpansive; 4) $T$ is not quasi-nonexpansive;  5) $T$ is not strictly pseudocontractive. 
\end{ex}

For more details and a complete diagram of the relationships between the mappings introduced in Definition \ref{def1}, we also refer to \cite{Ber24}.

The next lemmas will be useful in proving our main results in the next section.

\begin{lem} [\cite{MarinoXu}, Lemma 1.1]\label{lem1}

For any $x,y\in H$, we have
\smallskip

$
(1) \, \|x+y\|^2 \leq \|x\|^2 +2 \langle y,x+y\rangle;
$
\medskip

$
(2) \|t x+(1-t)y\|^2=t\|x\|^2+(1-t)\|y\|^2-t(1-t)\|x-y\|^2,\,\forall t\in [0,1].
$

\end{lem}

Let $C$ be a closed convex subset $H$ . Then the {\it nearest point} ({\it metric}) {\it projection}  $P_C$ from $H$ onto $C$ assigns to each $x\in H$ its nearest point in $C$, denoted by $P_Cx$, that is, $P_C x$ is the unique point in $C$ with the property
\begin{equation}\label{2.1}
\|x-P_C x\|\leq \|x-y\|,\, \textnormal{ for all } x\in H.
\end{equation}

The metrical projection has many important properties, of which we collect the following ones

\begin{lem} [\cite{Xu11}, Proposition 3.1]\label{lem2}
Given $x\in H$ and $y\in C$, we have:

$
(i) \quad z=P_C x \textnormal{ if and only if } \langle x-z,y-z\rangle \leq 0, \forall y\in C;
$

$
(ii) \quad \|x-P_C x\|^2 \leq \|x-y\|^2-\|y-P_C y\|^2, \forall y\in C;
$

$
(iii) \quad \langle x-y,P_C x-P_C  y\rangle \leq\|x-P_C x\|^2 \geq \|P_C x-P_C y\|^2, \forall y\in C.
$

\end{lem}

\begin{remark}\label{rem1}
Property $(i)$ in Lemma \ref{lem2} shows that, for any $x\in H$, its projection on the closed convex set $C$  solves the variational inequality $ \langle x-z,y-z\rangle \leq 0, \forall y\in C$;

Property $(ii)$ in Lemma \ref{lem2} expresses the fact that $P_C$ is a {\it firmly nonexpansive} mapping, while property $(iii)$ shows that $P_C$ is $1$-inverse strongly monotone.

\end{remark}

\begin{lem} [\cite{Byrne04}]\label{lem3}
Let $f$ be given by \eqref{norma}. The $\nabla f$ is $\|A\|^2$-Lipschitzian.
\end{lem}

Denote, as usually, the weak convergence in $H$ by $\rightharpoonup$ and the strong convergence by $\rightarrow$. The next concept will be important in our considerations.

\begin{defn} \label{def2}
A mapping $S:C\rightarrow C$ is said to be demiclosed at $0$ in $C\subset H$ if, for any sequence $\{x_k\}$  in $C$, such that  $x_k \rightharpoonup x$,  and $Su_k\rightarrow 0$, we have $Sx = 0$.
\end{defn}

\begin{remark}
In the particular case $S=I-T$, then it follows that $x$ in Definition \ref{def2} is a fixed point of $T$.
\end{remark}

\begin{lem} [\cite{He13}, Lemma 7]\label{lem4}
Let $\{x_n\}$ be a sequence of nonnegative real numbers, for which we have
$$
x_{n+1}\leq (1-\Gamma_n) x_n+\Gamma_n \Lambda_n,\,n\geq 1,
$$
and
$$
x_{n+1}\leq  x_n-\Psi_n+\Phi_n,\,n\geq 1,
$$
where $\Gamma_n\in (0,1)$, $\Psi_n\subset [0,\infty)$, and $\{\Lambda_n\}$ and $\{\Phi_n\}$  are two sequences of real numbers with the following properties

$(i)$ $\sum\limits_{n=1}^{\infty} \Gamma_n=\infty$;\, $(ii)$ $\lim\limits_{n\rightarrow \infty} \Phi_n=0$;\,
$(iii)$ For any subsequence $\{n_k\}$ of $\{n\}$, $\lim\limits_{k\rightarrow \infty} \Psi_{n_k}\leq 0$ implies $\limsup\limits_{k\rightarrow \infty}\Lambda_{n_k}\leq 0$.
 
 Then $\lim\limits_{n\rightarrow \infty} x_n=0$.
 \end{lem}

\begin{lem}[\cite{Ber23}, Lemma 3.2]\label{lem5}
Let $H$ be a real Hilbert space, $C\subset H$ be a closed and convex set. If $T:C\rightarrow C$ is $k$-demicontractive, then for any $\lambda\in (0,1-k)$, $T_{\lambda}$ is quasi-nonexpansive. 
\end{lem}

\section{Main results}

In order to solve the $(SFP)$ \eqref{sfp}, we consider the following self-adaptive inertial algorithm.

{\bf Algorithm 1.}

{\bf Step 1.} Take $x_0,x_1\in H_1$ arbitrarily chosen; let $n:=1$;

{\bf Step 2.} Compute $x_n$ by means of the following formulas

\begin{equation}\label{3.1}
\begin{cases}
u_n:=x_n+\theta_n(x_n-x_{n-1})\\
y_n:=P_C \left((1-\delta_n) u_n-\tau_n A^*(I-P_Q)A u_n\right)+\delta_n S_{\lambda} u_n,\\
x_{n+1}:=\alpha_n g(x_n)+\beta_n u_n+\gamma_n y_n.
\end{cases}
\end{equation}
with $S_{\lambda}=(1-\lambda) I+\lambda S$, $\lambda\in (0,1)$, 

\begin{equation}\label{3.2}
\theta_n:=
\begin{cases}
\min\left\{\theta,\dfrac{\varepsilon_n}{\|x_n-x_{n-1}\|}\right\}, \textnormal{ if } x_n\neq x_{n-1}\\
\theta, \textnormal{ otherwise},
\end{cases}
\end{equation}
$\theta\geq 0$ is a given number, $\tau_n=\dfrac{\rho_n f(x_n)}{\|f(u_n)\|^2}$, where $f$ is given by \eqref{norma},  $\rho_n\in (0,4)$ and $\{\alpha_n\}$, $\{\beta_n\}$,  $\{\gamma_n\}$,  $\{\delta_n\}$  are sequences in $(0,1)$ satisfying the following conditions
\smallskip

$(c_1)$ $\limsup\limits_{n\rightarrow \infty} \beta_n<1$; $(c_2)$ $\lim\limits_{n\rightarrow \infty} \dfrac{\varepsilon_n}{\alpha_n}=0$; 

$(c_3)$ $\lim\limits_{n\rightarrow \infty} \alpha_n=0$ and $\sum\limits_{n=1}^{\infty}=+\infty$; 

$(c_4)$ $0<\liminf\limits_{n\rightarrow \infty} \delta_n\leq \limsup\limits_{n\rightarrow \infty} \delta_n<1$;

$(c_5)$  $\alpha_n+\beta_n+\gamma_n=1$, $n\geq 1$.
\medskip

{\bf Step 3.} If $\nabla f(u_n)=0$, then Stop, otherwise let $n:=n+1$ and go to Step 2.
\medskip

The next technical Lemmas will be useful in proving our main result in this paper.

\begin{lem} \label{lem6}
Let $S:H_1\rightarrow H_1$ be a $k$-demicontractive mapping and  $\{x_n\}$ be the sequence generated by Algorithm 1. If $x^*\in Fix\,(S)$, then the sequence $\{\|x_n-x^*\|\}$ is bounded.
\end{lem}
\begin{proof}
Since $S$ is $k$-demicontractive, by Lemma \ref{lem5} we deduce that the averaged mapping $S_{\lambda}=(1-\lambda) I+\lambda S$  is also quasi-nonexpansive, for any $\lambda\in (0,1-k)$, and that $Fix\,(S)=Fix\,(S_{\lambda})$, for any $\lambda\in (0,1]$ (see for example \cite{Ber23}). 

In the following, to simplify writing, we shall denote $S_{\lambda}$ by $T$. So, $Fix\,(T)\neq \emptyset$ and let $x^*\in Fix\,(T)\cap SFP(C,Q)$. Let $y_n$ be defined by \eqref{3.1}. 

Then, by using \eqref{norma} and \eqref{nabla}, Lemma \ref{lem1} and Lemma \ref{lem2} and exploiting the fact that $T$ is quasi-nonexpansive, we have successively
$$
\|y_n-x^*\|^2=\|P_C\left((1-\delta_n)(u_n-\tau_n A^*(I-P_Q)Au_n\right)+\delta_n T u_n)-x^*\|^2
$$
$$
\leq \|\left((1-\delta_n)u_n-\tau_n A^*(I-P_Q)Au_n\right)+\delta_n T u_n)-x^*\|^2
$$
$$
-\|(I-P_C)\left((1-\delta_n)(u_n-\tau_n A^*(I-P_Q)Au_n)+\delta_n T u_n\right)\|^2
$$
$$
=\|\delta_n(T u_n-x^*)+(1-\delta_n)(u_n-\tau_n A^*(I-P_Q)Au_n-x^*)\|^2
$$
$$
-\|(I-P_C)\left((1-\delta_n)(u_n-\tau_n A^*(I-P_Q)Au_n)+\delta_n T u_n\right)\|^2
$$
$$
\leq \delta_n\| u_n-x^*\|^2+(1-\delta_n)\|u_n-\tau_n \nabla f(u_n)-x^*\|^2
$$
$$
-\delta_n (1-\delta_n)\|T u_n -u_n+\tau_n \nabla f(u_n)\|^2
$$
$$
-\|(I-P_C)\left((1-\delta_n)(u_n-\tau_n \nabla f(u_n))+\delta_n T u_n\right)\|^2
$$
$$
\leq \delta_n\| u_n-x^*\|^2+(1-\delta_n)(\|u_n-x^*\|^2+\tau_n^2 \|\nabla f(u_n)\|^2
$$
$$
-2\tau_n\langle \nabla f(u_n), u_n-x^*\rangle)-\delta_n (1-\delta_n)\|T u_n -u_n+\tau_n A^*(I-P_Q)Au_n\|^2
$$
\begin{equation}\label{3.3}
-\|(I-P_C)\left((1-\delta_n)(u_n-\tau_n \nabla f(u_n))+\delta_n T u_n\right)\|^2.
\end{equation}
On the other hand
$$
\langle \nabla f(u_n), u_n-x^*\rangle=\langle A^*(I-P_Q)Au_n, u_n-x^*\rangle
$$
\begin{equation}\label{3.4}
=\langle (I-P_Q)Au_n-(I-P_C) A x^*, Au_n-Ax^*\rangle\geq \|(I-P_Q)Au_n\|^2=2f(u_n).
\end{equation}
So, by inserting \eqref{3.4} in \eqref{3.3} we obtain
$$
\|y_n-x^*\|^2\leq \| u_n-x^*\|^2-4(1-\delta_n) \tau_n f(x_n)+(1-\delta_n)\tau_n^2 \|\nabla f(u_n)\|^2
$$
$$
-\delta_n (1-\delta_n)\|T u_n -u_n+\tau_n A^*(I-P_Q) Au_{n_k}\|^2
$$
$$
-\|(I-P_C)\left((1-\delta_n)(u_n-\tau_n \nabla f(u_n))+\delta_n T u_n\right)\|^2.
$$
$$
=\| u_n-x^*\|^2-(1-\delta_n)\rho_n(4-\rho_n)\cdot \frac{f^2(u_n)}{\|\nabla f(u_n)\|^2}
$$
$$
-\delta_n (1-\delta_n)\|T u_n -u_n+\tau_n A^*(I-P_Q) Au_{n_k}\|^2
$$
\begin{equation}\label{3.5}
-\|(I-P_C)\left((1-\delta_n)(u_n-\tau_n \nabla f(u_n))+\delta_n T u_n\right)\|^2.
\end{equation}
Now, having in view that $\rho_n\in (0,4)$ and $\delta_n\in (0,1)$, by \eqref{3.3} and \eqref{3.5} we deduce that
\begin{equation}\label{3.6}
\|y_n-x^*\|\leq \|u_n-x^*\|.
\end{equation}
Denote
\begin{equation}\label{veul}
v_n:=\dfrac{1}{1-\alpha_n}(\beta_n u_n+\gamma_n y_n)
\end{equation}
and apply Lemma \ref{lem1}, by keeping in mind condition $(c_5)$, to get
$$
\|v_n-x^*\|^2=\left\|\frac{\beta_n}{1-\alpha_n} u_n+\frac{\gamma_n}{1-\alpha_n}y_n-x^*\right\|^2
$$
$$
=\left\|\frac{\beta_n}{1-\alpha_n} (u_n-x^*)+\frac{\gamma_n}{1-\alpha_n}(y_n-x^*)\right\|^2
$$
\begin{equation}\label{3.7}
=\frac{\beta_n}{1-\alpha_n} \|u_n-x^*\|^2+\frac{\gamma_n}{1-\alpha_n}\|y_n-x^*\|^2-\frac{\beta_n}{1-\alpha_n} \cdot \frac{\gamma_n}{1-\alpha_n}\|u_n-y_n\|^2.
\end{equation}
Now, using the assumptions $(c_1)$-$(c_5)$, from the above inequality we get
$$
\|v_n-x^*\|^2\leq \frac{\beta_n}{1-\alpha_n} \|u_n-x^*\|^2+\frac{\gamma_n}{1-\alpha_n}\|y_n-x^*\|^2
$$
which, by using \eqref{3.6} and \eqref{3.5}, yields
 $$
\|v_n-x^*\|^2\leq \frac{\beta_n}{1-\alpha_n} \|u_n-x^*\|^2+\frac{\gamma_n}{1-\alpha_n}\|u_n-x^*\|^2
$$
$$
-(1-\delta_n)\rho_n(4-\rho_n)\cdot \frac{\gamma_n}{1-\alpha_n}\cdot \frac{f^2(u_n)}{\|\nabla f(u_n)\|^2}
$$
$$
-\delta_n (1-\delta_n)\frac{\gamma_n}{1-\alpha_n}\cdot\|T u_n -u_n+\tau_n \nabla f(u_n)\|^2
$$
$$
-\frac{\gamma_n}{1-\alpha_n}\cdot \|(I-P_C)\left((1-\delta_n)(u_n-\tau_n \nabla f(u_n))+\delta_n T u_n\right)\|^2
$$
$$
=\|u_n-x^*\|^2-(1-\delta_n)\rho_n(4-\rho_n)\cdot \frac{\gamma_n}{1-\alpha_n}\cdot \frac{f^2(u_n)}{\|\nabla f(u_n)\|^2}
$$
$$
-\delta_n (1-\delta_n)\frac{\gamma_n}{1-\alpha_n}\cdot\|T u_n -u_n+\tau_n \nabla f(u_n)\|^2
$$
\begin{equation}\label{eq-14}
-\frac{\gamma_n}{1-\alpha_n}\cdot \|(I-P_C)\left((1-\delta_n)(u_n-\tau_n \nabla f(u_n))+\delta_n T u_n\right)\|^2.
\end{equation}
The previous inequality implies
\begin{equation}\label{3.8}
\|v_n-x^*\|\leq \|u_n-x^*\|,\,n\geq 1.
\end{equation}
By \eqref{veul}, $(c_5)$ and the third equation in \eqref{3.1}, we obtain
\begin{equation}\label{3.8a}
x_{n+1}=\alpha_n g(x_n)+(1-\alpha_n) v_n,\,n\geq 1
\end{equation}
and so
$$
\|x_{n+1}-x^*\|=\|\alpha_n g(x_n)+(1-\alpha_n) v_n-x^*\|
$$
$$
=\|\alpha_n (g(x_n)-x^*)+(1-\alpha_n) (v_n-x^*)\|
$$
$$
\leq \alpha_n\| g(x_n)-x^*\|+(1-\alpha_n) \|v_n-x^*\|
$$
$$
\leq \alpha_n\| g(x_n)-g(x^*)\|+\alpha_n\| g(x^*)-x^*\|+(1-\alpha_n) \|v_n-x^*\|.
$$
Now, using the fact that $g$ is a $c$-contraction, we have
 $$
\|x_{n+1}-x^*\|\leq \alpha_n c\|x_n-x^*\|+\alpha_n\| g(x^*)-x^*\|
$$
$$
+(1-\alpha_n) \|x_n-x^*+\theta_n(x_n-x_{n-1})\|
$$
$$
\leq \alpha_n c\|x_n-x^*\|+\alpha_n\| g(x^*)-x^*\|+(1-\alpha_n) \|x_n-x^*\|+(1-\alpha_n)\theta_n \|x_n-x_{n-1}\|
$$
$$
\leq (1-\alpha_n (1-c))\|x_n-x^*\|+\alpha_n\| g(x^*)-x^*\|+\theta_n \|x_n-x_{n-1}\|.
$$
Denote $\varepsilon_n:=\theta_n \|x_n-x_{n-1}\|$. Then by the previous inequalities we get
$$
\|x_{n+1}-x^*\|\leq (1-\alpha_n (1-c))\|x_n-x^*\|
$$
\begin{equation}\label{3.9}
+\alpha_n (1-c)\left(\frac{\| g(x^*)-x^*\|}{1-c}+\frac{\varepsilon_n}{\alpha_n (1-c)}\right).
\end{equation}
Having in mind assumption $(c_2)$, take $M>0$ for which $\dfrac{\varepsilon_n}{\alpha_n}\leq M$, for all $n\geq 1$. Then, by denoting $M_1:=\dfrac{\| g(x^*)-x^*\|+M}{1-c}$, the inequality \eqref{3.9} yields
$$
\|x_{n+1}-x^*\|\leq (1-\alpha_n (1-c))\|x_n-x^*\|+\alpha_n (1-c)M_1
$$
$$
\leq \max\left\{\|x_n-x^*\|, M_1\right\},
$$
from which we easily obtain
\begin{equation}\label{3.10}
\|x_{n+1}-x^*\|\leq \max\left\{\|x_n-x^*\|, M_1\right\},\,n\geq 1,
\end{equation}
and this shows that $\{\|x_{n}-x^*\|\}$ is bounded.
\end{proof}

\begin{lem} \label{lem7}
Let $S:H_1\rightarrow H_1$ be a $k$-demicontractive mapping such that $I-T$ is demiclosed at zero, $g:H_1\rightarrow H_1$ is a $c$-Banach contraction and suppose that  $\{\alpha_n\}$, $\{\beta_n\}$,  $\{\gamma_n\}$,  $\{\delta_n\}$  are sequences in $(0,1)$ satisfying  conditions $(c_1)$-$(c_5)$ in Algorithm 1.
%\smallskip

%$(c_1)$ $\limsup\limits_{n\rightarrow \infty} \beta_n<1$; $(c_2)$ $\lim\limits_{n\rightarrow \infty} \dfrac{\varepsilon_n}{\alpha_n}=0$; 

%$(c_3)$ $\lim\limits_{n\rightarrow \infty} \alpha_n=0$ and $\sum\limits_{n=1}^{\infty}=+\infty$; 

%$(c_4)$ $0<\liminf\limits_{n\rightarrow \infty} \delta_n\leq \limsup\limits_{n\rightarrow \infty} \delta_n<1$;

%$(c_5)$  $\alpha_n+\beta_n+\gamma_n=1$, $n\geq 1$.
%\medskip

Let $x^*\in Fix\,(T)\cap SFP(C,Q)$, $\{x_n\}$ be the sequence generated by Algorithm 1, $f$ be defined by \eqref{norma} and let $\{v_n\}$ be the sequence given by \eqref{veul}. For $n\geq 1$, let us denote
$$
\Gamma_n:=2(1-c)\alpha_n;\, \Phi_n:=2\alpha_n\langle g(x_n)-v_n,x_{n+1}-x^*\rangle, 
$$
$$
\Lambda_n:=\frac{1}{2(1-c)}\left(\alpha_n\|g(x_n)-x^*\|^2+2\alpha_n \|g(x_n)-x^*\| \|v_n-x^*\|\right.
$$
$$
\left.+\alpha_n\|x_n-x^*\|^2+\frac{2\epsilon_n}{\alpha_n}\|v_n-x^*\|+2 \langle g(x^*)-x^*,v_{n}-x^*\rangle\right),
$$
and
$$
\Psi_n:=(1-\delta_n)\frac{\gamma_n}{1-\alpha_n}\rho_n(4-\rho)\frac{f^2(u_n)}{\|\nabla f(u_n)\|^2}
$$
$$
+\delta_n (1-\delta_n)\frac{\gamma_n}{1-\alpha_n}\cdot\|T u_n -u_n+\tau_n \nabla f(u_n)\|^2
$$
\begin{equation}\label{psi}
+\frac{\gamma_n}{1-\alpha_n}\cdot \|(I-P_C)\left((1-\delta_n)(u_n-\tau_n \nabla f(u_n))+\delta_n T u_n\right)\|^2.
\end{equation}

Then, for any subsequence $\{n_k\}$ of $\{n\}$, we have
\begin{equation}\label{3.11}
\limsup_{k\rightarrow \infty} \Lambda_{n_k}\leq 0,\, 
\end{equation}
whenever
\begin{equation}\label{3.12}
\lim_{k\rightarrow \infty} \Psi_{n_k}= 0.
\end{equation}
\end{lem}
\begin{proof}
Assume \eqref{3.12} holds. Then, by \eqref{psi} one deduces that all terms in the expression of $\Psi_{n_k}$ tend to zero as $k\rightarrow \infty$. So,
$$
\lim_{k\rightarrow \infty} \rho_{n_k}(4-\rho_{n_k})\frac{f^2(u_{n_k})}{\|\nabla f(u_{n_k})\|^2}=0
$$
and, based on assumptions $(c_1)$-$(c_5)$, it follows that in fact
\begin{equation}\label{3.13}
\lim_{k\rightarrow \infty} \frac{f^2(u_{n_k})}{\|\nabla f(u_{n_k})\|^2}=0.
\end{equation}
On the other hand since, by Lemma \ref{lem3},  $\nabla f(u_{n_k})$ is Lipschitzian, it follows that $\|\nabla f(u_{n_k})\|$ is bounded and therefore by \eqref{3.13} we deduce that $f(u_{n_k})\rightarrow 0$ as $k\rightarrow \infty$. which implies that 
$$
\lim_{k\rightarrow \infty} \|(I-P_Q) Au_{n_k}\|=0.
$$

By \eqref{3.12} we also get
\begin{equation}\label{3.14}
\lim_{k\rightarrow \infty} \|T u_{n_k} -u_{n_k}+\tau_{n_k} A^*(I-P_Q) Au_{n_k}\|^2=0
\end{equation}
and due to the fact that
\begin{equation}\label{3.14a}
\lim_{k\rightarrow \infty} \tau_{n_k}\|\nabla f(u_{n_k})\|=\lim_{k\rightarrow \infty} \frac{\rho_{n_k}f(u_{n_k}}{\|\nabla f(u_{n_k})\|}=0,
\end{equation}
we obtain
\begin{equation}\label{3.15}
\lim_{k\rightarrow \infty} \|T u_{n_k} -u_{n_k}\|=0.
\end{equation}
On the other hand, by \eqref{3.12} we also obtain

\begin{equation}\label{3.16}
\lim_{k\rightarrow \infty} \|(I-P_C)\left((1-\delta_{n_k})(u_{n_k}-\tau_{n_k} \nabla f(u_{n_k}))+\delta_{n_k} T u_{n_k}\right)\|=0
\end{equation}
which, by using the definition of $y_{n_k}$, yields
$$
\lim_{k\rightarrow \infty} \|(1-\delta_{n_k}) (u_{n_k}-\tau_{n_k}\nabla f(u_{n_k}))+\delta_{n_k} T u_{n_k}-y_{n_k}\|=0
$$
and this can be written in the expanded form
\begin{equation}\label{3.17}
\lim_{k\rightarrow \infty} \|(1-\delta_{n_k}) u_{n_k}-(1-\delta_{n_k})\tau_{n_k}\nabla f(u_{n_k}))+\delta_{n_k} T u_{n_k}-y_{n_k}\|=0.
\end{equation}
By \eqref{3.17} and \eqref{3.14a} we get
$$
\lim_{k\rightarrow \infty} \|(1-\delta_{n_k}) u_{n_k}+\delta_{n_k} T u_{n_k}-y_{n_k}\|=0
$$
which means that
\begin{equation}\label{3.18}
\lim_{k\rightarrow \infty} \|u_{n_k}-y_{n_k}+\delta_{n_k} (T u_{n_k}-u_{n_k})\|=0.
\end{equation}
Now, using the fact that
$$
\|u_{n_k}-y_{n_k}\|=\|u_{n_k}-y_{n_k}+\delta_{n_k} (T u_{n_k}-u_{n_k})-\delta_{n_k} (T u_{n_k}-u_{n_k})\|
$$
$$
\leq \|u_{n_k}-y_{n_k}+\delta_{n_k} (T u_{n_k}-u_{n_k})\|+\delta_{n_k} \|T u_{n_k}-u_{n_k}\|,
$$
by \eqref{3.15} and \eqref{3.18} one immediately obtain
\begin{equation}\label{3.19}
\lim_{k\rightarrow \infty} \|u_{n_k}-y_{n_k}\|=0.
\end{equation}
By using the definition of $v_n$ in \eqref{veul}, we have
$$
\|v_{n_k}-u_{n_k}\|=\left\|\frac{\beta_{n_k}}{1-\alpha_{n_k}} u_{n_k}+\frac{\gamma_{n_k}}{1-\alpha_{n_k}}y_{n_k}-u_{n_k}\right\|
$$
$$
=\left\|\frac{\gamma_{n_k}}{1-\alpha_{n_k}} u_{n_k}+\frac{\gamma_{n_k}}{1-\alpha_{n_k}}y_{n_k}\right\|
$$
$$
=\frac{\gamma_{n_k}}{1-\alpha_{n_k}}\cdot \|y_{n_k}-u_{n_k}\|
$$
which, by \eqref{3.19}, yields
\begin{equation}\label{3.20}
\lim_{k\rightarrow \infty} \|v_{n_k}-u_{n_k}\|=0.
\end{equation}
Since $I-S$ is demiclosed at zero and $T=(1-\lambda)I+\lambda T$, it follows that $T$ is also demiclosed at zero. By means of \eqref{3.15}, this implies that $\omega_w(u_{n_k})\subset Fix\,(T)$.

So, we can choose a subsequence $u_{n_{k_j}}$ of $u_{n_{k}}$ with the following property
$$
\limsup_{k\rightarrow \infty}\langle g(x^*)-x^*,u_{n_{k}}-x^*\rangle=\lim_{j\rightarrow \infty}\langle g(x^*)-x^*,u_{n_{k_j}}-x^*\rangle.
$$
We can assume, without any loss of generality that for the above subsequence $u_{n_{k_j}}$ one has  $u_{n_{k_j}}\rightharpoonup u'$. 

Since $f(u_{n_k})\rightarrow 0$, one obtains
$$
0\leq f(u')\leq \liminf\limits_{j\rightarrow \infty} f(u_{n_{k_j}})=0
$$
which implies $f(u')=0$ and $Au'\in Q$. 

Therefore, by \eqref{3.19}, $u'\in SFP(C,Q)$ and so $u'\in Fix\,(T)\cap SFP(C,Q)$. Now, based on \eqref{3.20}, we get
$$
\limsup_{k\rightarrow \infty}\langle g(x^*)-x^*,v_{n_{k}}-x^*\rangle=\limsup_{k\rightarrow \infty}\langle g(x^*)-x^*,u_{n_{k}}-x^*\rangle
$$
$$
=\lim_{j\rightarrow \infty}\langle g(x^*)-x^*,u_{n_{k_j}}-x^*\rangle=\langle g(x^*)-x^*,u'-x^*\rangle\leq 0,
$$
which shows that  \eqref{3.11} holds.
\end{proof}

Now we are ready to state and prove the main result of our paper.

\begin{theorem}\label{th1}
Let $T:H_1\rightarrow H_1$ be a $k$-demicontractive mapping such that $I-T$ is demiclosed at zero, and $g:H_1\rightarrow H_1$ be a $c$-Banach contraction. Suppose that  $\{\alpha_n\}$, $\{\beta_n\}$,  $\{\gamma_n\}$,  $\{\delta_n\}$  are sequences in $(0,1)$ satisfying conditions $(c_1)$-$(c_5)$ in Lemma \ref{lem7}.
\medskip

If $Fix\,(T)\cap SFP(C,Q)\neq \emptyset$, then the sequence $\{x_n\}$ generated by Algorithm 1 converges strongly to an element $x^*\in Fix\,(T)\cap SFP(C,Q)$ which solves uniquely the variational inequality \eqref{1.7}.

\end{theorem}

\begin{proof}
Using the fact that any metric projection is nonexpansive, by one hand, and that the composition of a nonexpansive mapping and of a contraction is a contraction, too, on the other hand, it follows that $P_{Fix\,(T)\cap SFP(C,Q)} g$ is a $c$-contraction, since $g$ is a $c$-contraction.
% and any metric projection is nonexpansive, it follows that the composite $P_{Fix\,(T)\cap SFP(C,Q)} g$ is also a $c$-contraction. 

Hence $P_{Fix\,(T)\cap SFP(C,Q)} g$ has a unique fixed point $x^*\in H_1$:
$$
x^*=P_{Fix\,(T)\cap SFP(C,Q)} g(x^*).
$$% such that $x^*$  is a fixed point of $g$.

Moreover, in view of Lemma \ref{lem2},  $x^*\in Fix\,(T)\cap SFP(C,Q)$ is a solution of the variational inequality \eqref{1.7}.% if and only if 

Let $p\in Fix\,(T)\cap SFP(C,Q)$ be arbitrary. By Lemma \ref{lem6}, it follows that the sequence $\{\|x_{n}-x^*\|\}$ is bounded.

Let $\{u_n\}$ be given by the corresponding inertial equation in \eqref{3.1}. The we have
$$
\|u_n-p\|^2=\|x_n+\theta_n(x_n-x_{n-1})\|^2
$$
which by applying Lemma \ref{lem1} yields
$$
\|u_n-p\|^2\leq \|x_n-p\|^2+2\theta_n\langle x_n-x_{n-1}, u_n-p\rangle
$$
$$
\leq \|x_n-p\|^2+2\theta_n \|x_n-x_{n-1}\|\|u_n-p\|\leq \|x_n-p\|^2+2\varepsilon_n \|u_n-p\|.
$$
Hence
\begin{equation}\label{3.21}
\|u_n-p\|^2\leq \|x_n-p\|^2+2\varepsilon_n \|u_n-p\|.
\end{equation}
As $p\in Fix\,(T)\cap SFP(C,Q)$ has been taken arbitrarily, we can let it to be $$p:=x^*=P_{Fix\,(T)\cap SFP(C,Q)} g(x^*).$$ By using \eqref{3.8a}, we have
$$
\|x_{n+1}-x^*\|^2=\|\alpha_n g(x_n)+(1-\alpha_n) v_n-x^*\|^2
$$
$$
=\|\alpha_n (g(x_n)-x^*)+(1-\alpha_n) (v_n-x^*)\|^2\leq \alpha_n^2 \|g(x_n)-x^*\|^2
$$
$$
+(1-\alpha_n)^2 \|v_n-x^*\|^2+2\alpha_n \langle g(x_n)-x^*, v_n-x^*\rangle -2\alpha_n^2 \langle g(x_n)-x^*, v_n-x^*\rangle
$$
$$
\leq \alpha_n^2 \|g(x_n)-x^*\|^2+(1-\alpha_n)^2 \|v_n-x^*\|^2+2\alpha_n \langle g(x_n)-x^*, v_n-x^*\rangle
$$
$$
+2\alpha_n^2 \|g(x_n)-x^*\|\| v_n-x^*\|=\alpha_n^2 \|g(x_n)-x^*\|^2+(1-\alpha_n)^2 \|v_n-x^*\|^2
$$
$$
+2\alpha_n \langle g(x_n)-g(x^*), v_n-x^*\rangle+2\alpha_n \langle g(x^*)-x^*, v_n-x^*\rangle
$$
$$
+2\alpha_n^2 \|g(x_n)-x^*\|\| v_n-x^*\|
$$
$$
\leq \alpha_n^2 \|g(x_n)-x^*\|^2+2\alpha_n^2 \|g(x_n)-x^*\|\| v_n-x^*\|+(1-\alpha_n)^2 \|v_n-x^*\|^2%+\alpha_n\cdot c\cdot \left(\|x_{n}-x^*\|^2+\|v_n-x^*\|^2\right)
$$
\begin{equation}\label{eq-18}
+\alpha_n\cdot c\cdot \left(\|x_{n}-x^*\|^2+\|v_n-x^*\|^2\right)+2\alpha_n \langle g(x^*)-x^*, v_n-x^*\rangle.
\end{equation}
Now, by \eqref{3.6} and \eqref{3.21} one obtains
\begin{equation}\label{eq-19}
\|v_n-x^*\|^2\leq \|x_n-x^*\|^2+2\varepsilon_n \|u_n-x^*\|,
\end{equation}
and so, by using \eqref{eq-18}, we deduce that
$$
\|x_{n+1}-x^*\|^2\leq \alpha_n^2 \|g(x_n)-x^*\|^2+2\alpha_n^2 \|g(x_n)-x^*\|\| v_n-x^*\|
$$
$$
+\alpha_n\cdot c\cdot \left(\|x_{n}-x^*\|^2+\|v_n-x^*\|^2\right)+2\varepsilon_n \|u_n-x^*\|+2\alpha_n \langle g(x^*)-x^*, v_n-x^*\rangle
$$
$$
+(1-\alpha_n)^2\cdot \left(\|x_{n}-x^*\|^2+2\varepsilon_n \|u_n-x^*\|\right)
$$
$$
=\alpha_n^2 \|g(x_n)-x^*\|^2+2\alpha_n^2 \|g(x_n)-x^*\|\| v_n-x^*\|
$$
$$
+\left(\alpha_n^2+(1-2\alpha_n (1-c))\right)\cdot \|x_{n}-x^*\|^2+\left(2\varepsilon_n (1-\alpha_n)^2+ 2\alpha_n c \varepsilon_n\right)\cdot \|u_n-x^*\|
$$
$$
+2\alpha_n \|g(x_n)-x^*\|\| v_n-x^*\|\leq \alpha_n^2 \|g(x_n)-x^*\|^2+2\alpha_n^2 \|g(x_n)-x^*\|\| v_n-x^*\|
$$
$$
+\alpha_n^2 \cdot \|x_{n}-x^*\|^2+\left(1-2 \alpha_n (1-c)\right) \|x_{n}-x^*\|^2
$$
$$
+4\varepsilon_n\cdot \left(\alpha_n \|g(x_n)-x^*\|^2+2 \alpha_n  \|g(x_n)-x^*\|\| v_n-x^*\|+\alpha_n  \|x_n-x^*\|^2\right.
$$
$$
\left. +\frac{4\varepsilon_n}{\alpha_n}\cdot \|u_n-x^*\|+2\langle g(x^*)-x^*, v_n-x^*\rangle\right)=\left(1-2 \alpha_n (1-c)\right) \|x_{n}-x^*\|^2
$$
$$
+2\alpha_n (1-c)\frac{1}{2 (1-c)}\left[\alpha_n \|g(x_n)-x^*\|^2+2 \alpha_n  \|g(x_n)-x^*\|\| v_n-x^*\|\right.
$$
\begin{equation}\label{eq-20}
\left. +\alpha_n  \|x_n-x^*\|^2+\frac{4\varepsilon_n}{\alpha_n}\cdot \|u_n-x^*\|+2\langle g(x^*)-x^*, v_n-x^*\rangle\right].
\end{equation}
On the other hand, by Lemma \ref{lem1} and the definition of $\{x_n\}$ we have
%$$
%\|x_{n+1}-x^*\|^2=\|\alpha_n g(x_n)+(1-\alpha_n) v_n-x^*\|^2
%$$
\begin{equation}\label{eq-21}
\|x_{n+1}-x^*\|^2=\|\alpha_n g(x_n)+(1-\alpha_n) v_n-x^*\|^2\leq \|v_n-x^*\|^2+2 \alpha_n \langle g(x_n)-v_n, v_n-x^*\rangle.
\end{equation}
Therefore, by combining \eqref{eq-14}, \eqref{3.21} and \eqref{eq-21}, and denoting $S=(1-\lambda) I+\lambda T$, we obtain successively
$$
\|x_{n+1}-x^*\|^2\leq \|u_{n}-x^*\|^2-(1-\delta_n)\cdot \frac{\gamma_n}{1-\alpha_n}\cdot \rho_n (4-\rho_n)\cdot \frac{f^2(u_n)}{\|\nabla f(u_n)\|^2}
$$
$$
-\frac{\gamma_n}{1-\alpha_n}\cdot\delta_n(1-\delta_n)\cdot\|S u_n-u_n+\tau_n A^*(I-P_Q)A u_n\|^2
$$
$$
-\frac{\gamma_n}{1-\alpha_n}\cdot\|(I-P_C)\left((1-\delta_n)(u_n-\tau_n A^*(I-P_Q)A u_n\right)+\delta_n Su_n\|^2
$$
$$
+2\alpha_n\cdot  \langle g(x_n)-v_n, v_n-x^*\rangle\leq \|x_{n}-x^*\|^2+2 \varepsilon_n \|u_n-x^*\|^2
$$
$$
-(1-\delta_n)\cdot \frac{\gamma_n}{1-\alpha_n}\cdot \rho_n (4-\rho_n)\cdot \frac{f^2(u_n)}{\|\nabla f(u_n)\|^2}
$$
$$
-\frac{\gamma_n}{1-\alpha_n}\cdot\delta_n(1-\delta_n)\cdot\|S u_n-u_n+\tau_n A^*(I-P_Q)A u_n\|^2
$$
$$
-\frac{\gamma_n}{1-\alpha_n}\cdot\|(I-P_C)\left((1-\delta_n)(u_n-\tau_n A^*(I-P_Q)A u_n\right)+\delta_n Su_n\|^2
$$
$$
+2\alpha_n\cdot  \langle g(x_n)-v_n, v_n-x^*\rangle
$$
which yields
$$
\|x_{n+1}-x^*\|^2\leq \|x_{n}-x^*\|^2+2 \varepsilon_n \|u_n-x^*\|^2-(1-\delta_n)\cdot \frac{\gamma_n}{1-\alpha_n}\cdot \rho_n (4-\rho_n)\cdot \frac{f^2(u_n)}{\|\nabla f(u_n)\|^2}
$$
$$
-\frac{\gamma_n}{1-\alpha_n}\cdot\delta_n(1-\delta_n)\cdot\|S u_n-u_n+\tau_n A^*(I-P_Q)A u_n\|^2
$$
$$
-\frac{\gamma_n}{1-\alpha_n}\cdot\|(I-P_C)\left((1-\delta_n)(u_n-\tau_n A^*(I-P_Q)A u_n\right)+\delta_n Su_n\|^2
$$
\begin{equation}\label{eq-22}
+2\alpha_n\cdot  \langle g(x_n)-v_n, v_n-x^*\rangle.
\end{equation}
Now, for $n\geq 1$, let us denote
$$
\Gamma_n:=2(1-c)\alpha_n;\, \Phi_n:=2\alpha_n\langle g(x_n)-v_n,x_{n+1}-x^*\rangle, 
$$
$$
\Lambda_n:=\frac{1}{2(1-c)}\left(\alpha_n\|g(x_n)-x^*\|^2+2\alpha_n \|g(x_n)-x^*\| \|v_n-x^*\|\right.
$$
$$
\left.+\alpha_n\|x_n-x^*\|^2+\frac{2\epsilon_n}{\alpha_n}\|v_n-x^*\|+2 \langle g(x^*)-x^*,v_{n}-x^*\rangle\right),
$$
and
$$
\Psi_n:=(1-\delta_n)\frac{\gamma_n}{1-\alpha_n}\rho_n(4-\rho)\frac{f^2(u_n)}{\|\nabla f(u_n)\|^2}
$$
$$
+\delta_n (1-\delta_n)\frac{\gamma_n}{1-\alpha_n}\cdot\|T u_n -u_n+\tau_n \nabla f(u_n)\|^2
$$
$$
+\frac{\gamma_n}{1-\alpha_n}\cdot \|(I-P_C)\left((1-\delta_n)(u_n-\tau_n \nabla f(u_n))+\delta_n T u_n\right)\|^2.
$$
In view of these notations, inequalities \eqref{eq-20} and  \eqref{eq-22} can be briefly written as 
$$
\|x_{n+1}-x^*\|^2\leq (1-\Gamma_n) \|x_{n}-x^*\|^2+\Gamma_n \Lambda_n,\,n\geq 1
$$
$$
\|x_{n+1}-x^*\|^2\leq \|x_{n}-x^*\|^2-\Psi_n+\Phi_n,\,n\geq 1,
$$
respectively.

By assumptions $(c_1)-(c_5)$ it is easy to deduce that 
$$
\lim_{n\rightarrow\infty} \Gamma_n=0,\,\lim_{n\rightarrow\infty}\Phi_n=0 \textnormal{ and } \sum_{n=0}^{\infty} \Gamma_n=\infty. 
$$
In the end, by applying Lemma \ref{lem7} and Lemma \ref{lem4}, it follows that
$$
\lim_{n\rightarrow\infty} \|x_n-x^*\|=0
$$
which shows that the sequence $\{x_n\}$ generated by Algorithm 1 converges strongly to $x^*$.
\end{proof}

\begin{remark} 
We note that the technique of proof of Theorem  \ref{th1}  is similar to that used in \cite{Wang-Xu} and is based on inserting an averaged component which produces a perturbed version of the inertial algorithm, thus imbedding the demicontractive mappings in the class of quasi nonexpansive mappings, in view  of Lemma \ref{lem5}.
 \end{remark}

\begin{ex} \label{ex2}
Let $H$ be the real line with the usual norm, $C = [0, 1]$ and $T$ be the mapping in Example \ref{ex1}. Since $T$ is demicontractive, our Theorem \ref{th1} can be applied to solve any consistent split feasibility problem over the set of fixed points of $T$, whenever $Fix\,(T)\cap SFP(C,Q)\neq \emptyset$.

We also note that  Theorem 2.1 in Qin and Wang \cite{Qin} cannot be applied to solve consistent split feasibility problems over the set of fixed points of $T$ (because $T$ is not nonexpansive) and also Theorem 1 in Wang et al. \cite{Wang-Xu} cannot be applied to the same problem (because $T$ is not quasi nonexpansive).
\end{ex}

\section{Numerical examples}

\begin{ex} \label{ex3}
We consider the problem given in Example 1 in Wang et al. \cite{Wang-Xu}, which is devoted to the solution of a linear system of equations $Ax=b$.  We work similarly in $H_1=H_2=\mathbb{R}^5$, with the same data, first by taking the mapping $S$ given by
$$
S=
\begin{pmatrix}
\frac{1}{3} & \frac{1}{3} & 0 & 0 & 0\\
0 & \frac{1}{3} & \frac{1}{3} & 0 & 0\\
0 & 0 & \frac{1}{3} & \frac{1}{3} & 0\\
0 & 0 & 0 & \frac{1}{3} & \frac{1}{3}\\
0 & 0 & 0 & 0 & 1
\end{pmatrix}
$$ 
and then considering a non viscosity type algorithm, i.e., taking the contraction mapping $g$ to be the null function $g\equiv 0$. To allow a numerical comparison, we also take
$$
A=
\begin{pmatrix}
1 & 1 & 2 & 2 & 1\\
0 & 2 & 1 & 5 & -1\\
1 & 1 & 0 & 4 & 1\\
2 & 0 & 3 & 1 & 5\\
2 & 2 & 3 & 6 & 1
\end{pmatrix},
\quad b=
\begin{pmatrix}
\frac{43}{16}\\
2\\
\frac{19}{16}\\
\frac{51}{8}\\
\frac{41}{8}
\end{pmatrix},
$$ 
\end{ex}
This is a particular example of a split feasibility problem with $C=Fix\,(S)$ and $Q=\{b\}$.

We performed several numerical experiments in MatLab by using  our Algorithm 1 with various particular values on the parameters  and compared the obtained results to those presented in Wang et al. \cite{Wang-Xu} (Table 1 and Figure 1). 

By analysing the diversity of the numerical results thus obtained, we noted a very interesting fact, i.e., that most of the assumptions on the parameters $\alpha_n$, $\beta_n$, $\delta_n$, $\theta_n$, $\tau_n$ involved in the iterative process \eqref{3.1} are in fact  imposed merely for technical reasons  when proving analytically the strong convergence of the sequence $\{x_n\}$ generated by Algorithm 1. 

Therefore, some of these assumptions appear to  be not necessary for the convergence of the iterative process $\{x_n\}$ in most practical situations, as shown by the numerical results presented in Table 1. 

These results were obtained for the same starting point $x$ like the one  in Wang et al. \cite{Wang-Xu} but with the following particular values of the involved parameters: $\beta_n=0$, $\delta_n=1$  (which do not satisfy all the assumptions in $(c_1)-(c_5)$), $\theta_n=0$ and $\lambda=0.5$.

It is also worth mentioning that we obtained the exact solution $x^*=(\dfrac{1}{16}, \dfrac{1}{8}, \dfrac{1}{4}, \dfrac{1}{2}, 1)$ of the problem after $n=33$ iterations.  

If we compare our numerical results to the results given Table 1 and Figure  in Wang et al. \cite{Wang-Xu}, where the authors have taken the values $\alpha_n=\frac{1}{10n}$, $\beta_n=0.5$, $\delta_n=0.5$,..., and the same starting point $x=\left(1, 1, 1, 1, 1\right)^T$ for  Algorithm (9), we observe that the exact solution was not obtained even after 10 000 iterations...

%There are a few reasons for these results. First, the mapping $S$ considered above is not nonexpansive (actually it is Lipschitzian continuous with Lipschitz constant $L=??$). Secondly, 

	\begin{table}[htbp]\resizebox{8cm}{!} 
{\begin{tabular}
{|c|@{}l|r|r|r|r|}
\hline
n& $x_n^{(1)}$& $x_n^{(2)}$& $x_n^{(3)}$& $x_n^{(4)}$&$x_n^{(5)}$\\
\hline\hline
0& 
1& 
1& 
1& 
1&
1\\
\hline

1& 
0.766667& 
0.766667& 
0.766667& 
0.766667&
1\\
\hline

2& 
0.587778&    
0.587778& 
0.587778& 
0.642222&
1\\
\hline

3& 
0.450630&    
0.450630& 
0.463333& 
0.575852&
1\\
\hline

4& 
0.345483&    
0.348447& 
0.381477& 
0.540454&
1\\
\hline

5& 
0.265562&    
0.274850& 
0.329560& 
0.521576&
1\\
\hline

6& 
0.205764&    
0.223484& 
0.297466& 
0.511507&
1\\
\hline

 7& 
0.161887&    
0.188600& 
0.278000& 
0.506137&
1\\
\hline

 8& 
0.130347&    
0.165454& 
0.266366& 
0.503273&
1\\
\hline

9& 
0.108124&    
0.150394& 
0.259492& 
0.501746&
1\\
\hline

10& 
0.092758&    
0.140758& 
0.255470& 
0.500931&
1\\
\hline

 11& 
0.082315&    
0.134681& 
0.253134& 
0.500497&
1\\
\hline

12& 
0.075327&    
0.130894& 
0.251788& 
0.500265&
1\\
\hline

 13& 
0.070716&    
0.128561& 
0.251015& 
0.500141&
1\\
\hline

14& 
0.067713&    
0.127136& 
0.250574& 
0.500075&
1\\
\hline

15& 
0.065779&    
0.126273& 
0.250324& 
0.500040&
1\\
\hline

...& 
...&    
...& 
...& 
...&
...\\
\hline

20&    
0.062790&    
0.125089& 
0.250018& 
0.500002&
1\\
\hline

...& 
...&    
...& 
...& 
...&
...\\
\hline

32&    
0.062501&       
0.125000& 
0.250000& 
0.500000&
1\\
\hline

33&    
0.062500&       
0.125000& 
0.250000& 
0.500000&
1\\
\hline

\hline
\end{tabular}}
%\label{tab2}
\end{table}

\medskip

Table 1. Numerical results for the starting point $x=(1, 1, 1, 1, 1)^T$

\medskip

In fact,   both algorithms considered in Wang et al. \cite{Wang-Xu}, i.e., algorithms (7) and (9), are extremely slow: even after performing 10000 iterations the exact solution $x^*$ is obtained with an error of $9.4925\times 10^{-5}$.

In our opinion, this is because any inertial type algorithm (i.e., with $\theta_n\neq 0$) is usually slower than the non inertial ones, as illustrated by the numerical examples in Table 1, see also the results reported in Berinde \cite{Ber22}, but for a slightly different context.

\section{Conclusions}

1. We introduced a hybrid inertial self-adaptive algorithm for solving the split feasibility problem and fixed point problem in the class of demicontractive mappings. 

2. As shown by Example \ref{ex2}, our theoretical results extend several related results existing in literature from the class of nonexpansive or quasi-nonexpansive mappings to the larger class of demicontractive mappings. 

3.  We performed numerical experiments, see Example \ref{ex3}, designed to compare our results to those presented in Wang et al. \cite{Wang-Xu}. The numerical results presented in Table 1 clearly illustrates the superiority of our  results  over the related existing ones in literature. These numerical results also naturally raise an open problem: find weaker conditions on the parameters $\alpha_n$, $\beta_n$, $\delta_n$, $\theta_n$, $\tau_n$  such that  the iterative process \eqref{3.1} still converges to an element $x^*\in Fix\,(T)\cap SFP(C,Q)$.

4. For other related works that allow similar developments to the ones in the current paper, we refer the readers to Kingkam and Nantadilok \cite{Kingkam}, Padcharoen et al. \cite{Pad}, Sharma and Chandok \cite{Sharma}, Shi et al. \cite{Shi}, Tiammee and Tiamme \cite{Tiam},...

\medskip

{\bf Acknowledgements}
\medskip

Not applicable.
\medskip

{\bf Funding}
\medskip

Not applicable.
\medskip

{\bf Availability of data and materials}
\medskip

No data were used to support this study.
\medskip

{\bf Competing interests}
\medskip

The author declares no competing interests.
\medskip

{\bf Author details}
\medskip

$^{1}$ Department of Mathematics and Computer
Science
North University Centre at Baia Mare
Technical University of Cluj-Napoca
Victoriei 76, 430072 Baia Mare Romania;
E-mail: vasile.berinde@mi.utcluj.ro
\medskip

$^{2}$ Academy of Romanian Scientists, 3 Ilfov, 050044, Bucharest, Romania;
E-mail: vasile.berinde@gmail.com


\begin{thebibliography}{99}

%\bibitem {Ber18} Berinde, V. Weak and strong convergence theorems for the Krasnoselskij iterative algorithm in the class of enriched strictly pseudocontractive operators. {\em An. Univ. Vest Timi\c s. Ser. Mat.-Inform.} {\bf 56} (2018), no. 2, 13--27.

%\bibitem {Ber19a} Berinde, V. Approximating fixed points of enriched nonexpansive mappings by Krasnoselskij iteration in Hilbert spaces. {\em Carpathian J. Math.} {\bf 35} (2019), no. 3, 293--304.

\bibitem {Ber20} Berinde, V. Approximating fixed points of enriched nonexpansive mappings in Banach spaces by using a retraction-displacement condition. {\em Carpathian J. Math.} {\bf 36} (2020), no. 1, 27--34.

\bibitem {Ber22} Berinde, V., A modified Krasnoselskii-Mann iterative algorithm for approximating fixed points of enriched nonexpansive mappings. {\em Symmetry} {\bf 2022}; 14(1):123. https://doi.org/10.3390/sym14010123.

\bibitem {Ber23} Berinde, V. Approximating fixed points of demicontractive mappings via the quasi-nonexpansive case. {\em Carpathian J. Math.} {\bf 39} (2023), no. 1, 73--85.

\bibitem {Ber24} Berinde, V. On a useful lemma that relates quasi-nonexpansive and demicontractive mappings in Hilbert spaces. {\em Creat. Math. Inform.} {\bf 33} (2024), no. 1, 7--21.


\bibitem{BPR23} Berinde, V.; Petru\c sel, A.; Rus, I. A. Remarks on the terminology of the mappings in fixed point iterative methods in metric spaces. {\em Fixed Point Theory} {\bf 24} (2023), no. 2, 525--540.

\bibitem{Byrne02} Byrne, C. Iterative oblique projection onto convex set and the split feasibility problem. {\it Inverse Probl.} {\bf 2002}, 18, 441--453.

\bibitem{Byrne04} Byrne, C. A unified treatment of some iterative algorithms in signal processing and image reconstruction. {\em Inverse Probl.} {\bf 2004}, 20, 103--120.

%\bibitem{Chen} Chen, H. Y.; Sahu, D. R.; Wong, N. C. Iterative algorithms for solving multiple split common fixed problems in Hilbert spaces. {\em J. Nonlinear Convex Anal.} {\bf 19} (2018), no. 2, 265--285.


%\bibitem {Cui} Cui, H. H.; Ceng, L. C.; Wang, F. H. Weak convergence theorems on the split common fixed point problem for demicontractive continuous mappings. {\em J. Funct. Spaces} {\bf 2018}, Art. ID 9610257, 7 pp.

%\bibitem {Cui} Cui, H. H. Multiple-sets split common fixed-point problems for demicontractive mappings. {\em J. Math.} {\bf 2021}, Art. ID 3962348, 6 pp.

%\bibitem {Dang} Dang, Y. Z.; Meng, F. W.; Sun, J. An iterative algorithm for split common fixed-point problem for demicontractive mappings. in {\em Optimization methods, theory and applications}, 85--94, Springer, Heidelberg, 2015.


%\bibitem{Fan21} Fan, Q.; Peng, J.; He, H. Weak and strong convergence theorems for the split common fixed point problem with demicontractive operators. {\em Optimization} {\bf 70} (2021), no. 5-6, 1409--1423.


%\bibitem {Han18} Hanjing, A.; Suantai, S. Solving split equality common fixed point problem for infinite families of demicontractive mappings. {\em Carpathian J. Math.} {\bf 34} (2018), no. 3, 321--331.

%\bibitem {Han18a} Hanjing, A.; Suantai, S.  The split common fixed point problem for infinite families of demicontractive mappings. {\em Fixed Point Theory Appl.} {\bf 2018}, Paper No. 14, 21 p.

%\bibitem {Han20} Hanjing, A.; Suantai, S. The split fixed point problem for demicontractive mappings and applications. {\em Fixed Point Theory}  {\bf 21} (2020), no. 2, 507--524.

%\bibitem {Han20a} Hanjing, A.; Suantai, S.  Hybrid inertial accelerated algorithms for split fixed point problems of demicontractive mappings and equilibrium problems. {\em Numer. Algorithms}  {\bf 85} (2020), no. 3, 1051--1073.

%\bibitem{Hanjing23} Hanjing, A.; Suantai, S.; Cho, Y. J. Hybrid inertial accelerated extragradient algorithms for split pseudomonotone equilibrium problems and fixed point problems of demicontractive mappings. {\em Filomat} {\bf 37} (2023), no. 5, 1607--1623.

%\bibitem{He21} He, H. M.; Peng, J.; Fan, Q. W. An iterative viscosity approximation method for the split common fixed-point problem. {\em Optimization} {\bf 70} (2021), no. 5-6, 1261--1274.

\bibitem{He13} He, S., Yang, C. Solving the variational inequality problem defined on intersection of finite level sets. {\em Abstr. Appl. Anal.} {\bf 2013},
2013, 942315.

\bibitem{Kingkam} Kingkam, P.; Nantadilok, J. Iterative process for finding fixed points of quasi-nonexpansive multimaps in CAT(0) spaces. {\em Korean J. Math.} {\bf 31} (2023), no. 1, 35--48.


%\bibitem{Jail18} Jailoka, P.; Suantai, S. Split null point problems and fixed point problems for demicontractive multivalued mappings. {\em Mediterr. J. Math.} {\bf 15} (2018), no. 5, Paper No. 204, 19 pp.

%\bibitem{Jail19} Jailoka, P.; Suantai, S. Split common fixed point and null point problems for demicontractive operators in Hilbert spaces. {\em Optim. Methods Softw.} {\bf 34} (2019), no. 2, 248--263.

%\bibitem{Jail19a} Jailoka, P.; Suantai, S. The split common fixed point problem for multivalued demicontractive mappings and its applications. {\em Rev. R. Acad. Cienc. Exactas F\' is. Nat. Ser. A Mat. RACSAM} {\bf 113} (2019), no. 2, 689--706.

%\bibitem{Jail21} Jailoka, P.; Suantai, S.  On split fixed point problems for multi-valued mappings and designing a self-adaptive method. {\em Results Math.} {\bf 76} (2021), no. 3, Paper No. 133, 21 pp.

\bibitem{Kraikaew} Kraikaew, R.; Saejung, S. A simple look at the method for solving split feasibility problems in Hilbert spaces. {\em Rev. R. Acad. Cienc. Exactas F\' is. Nat. Ser. A Mat. RACSAM} {\bf 114} (2020), no. 3, Paper No. 117, 9 pp.

\bibitem{Lopez} L\' opez, G.; Mart\' in-M\' arquez, V.; Wang, F. H.; Xu, H.-K. Solving the split feasibility problem without prior knowledge of matrix norms. {\it Inverse Problems} {\bf 28} (2012), no. 8, 085004, 18 pp.

\bibitem{MarinoXu} Marino, G.; Xu, H.-K. Weak and strong convergence theorems for strict pseudo-contractions in Hilbert spaces. {\em J. Math. Anal. Appl.} {\bf 329} (2007), no. 1, 336--346.

%\bibitem{Pad21} Padcharoen, A.; Kumam, P.; Cho, Y. J. Split common fixed point problems for demicontractive operators. {\em Numer. Algorithms} {\bf 82} (2019), no. 1, 297--320.

\bibitem{Pad} Padcharoen, A.; Sokhuma, K.; Abubakar, J. Projection methods for quasi-nonexpansive multivalued mappings in Hilbert spaces. {\em AIMS Math.} {\bf 8} (2023), no. 3, 7242--7257.



\bibitem {Qin} Qin, X. L.; Wang, L. A fixed point method for solving a split feasibility problem in Hilbert spaces. {\em Rev. R. Acad. Cienc. Exactas F\' is. Nat. Ser. A Mat. RACSAM} {\bf 113} (2019), no. 1, 315--325.

%\bibitem {QinW} Qin, L.-J.; Wang, G. Multiple-set split feasibility problems for a finite family of demicontractive mappings in Hilbert spaces. {\em Math. Inequal. Appl.} {\bf 16} (2013), no. 4, 1151--1157.


%\bibitem {Shehu16} Shehu, Y.; Cholamjiak, P. Another look at the split common fixed point problem for demicontractive operators. {\em Rev. R. Acad. Cienc. Exactas F\' is. Nat. Ser. A Mat. RACSAM} {\bf 110} (2016), no. 1, 201--218.


%\bibitem {Shehu16a} Shehu, Y.; Mewomo, O. T. Further investigation into split common fixed point problem for demicontractive operators. {\em Acta Math. Sin. (Engl. Ser.)} {\bf 32} (2016), no. 11, 1357--1376.


%\bibitem {Sua20} Suantai, S.; Jailoka, P. A self-adaptive algorithm for split null point problems and fixed point problems for demicontractive multivalued mappings. {\em Acta Appl. Math.} {\bf 170} (2020), 883--901. 


%\bibitem {Sua21} Suantai, S.;  Sarnmeta, P.; Chumpungam, D.; Inthakon, W. Split common fixed point problems for multi-valued demicontractive mappings in Hilbert spaces. {\em J. Nonlinear Convex Anal.} {\bf 22} (2021), no. 12, 2623--2637.

%\bibitem{Supa20} Suparatulatorn, R.; Cholamjiak, P.; Suantai, S. Self-adaptive algorithms with inertial effects for solving the split problem of the demicontractive operators. {\em Rev. R. Acad. Cienc. Exactas F\' is. Nat. Ser. A Mat. RACSAM} {\bf 114} (2020), no. 1, Paper No. 40, 16 pp.

%\bibitem{Supa21} Suparatulatorn, R.; Charoensawan, P.; Poochinapan, K.; Dangskul, S. An algorithm for the split feasible problem and image restoration. {\em Rev. R. Acad. Cienc. Exactas F\' is. Nat. Ser. A Mat. RACSAM} {\bf 115} (2021), no. 1, Paper No. 12, 18 pp.

%\bibitem{Tang12} Tang, Y.-C.; Peng, J.-G.; Liu, L.-W.  A cyclic algorithm for the split common fixed point problem of demicontractive mappings in Hilbert spaces. {\em Math. Model. Anal.} {\bf17} (2012), no. 4, 457--466.

%\bibitem{Tang14} Tang, Y.-C.; Peng, J.-G.; Liu, L.-W. A cyclic and simultaneous iterative algorithm for the multiple split common fixed point problem of demicontractive mappings. {\em Bull. Korean Math. Soc.} {\bf 51} (2014), no. 5, 1527--1538.


%\bibitem{Wang22} Wang, F. The split feasibility problem with multiple output sets for demicontractive mappings. {\em J. Optim. Theory Appl.} {\bf 195} (2022), no. 3, 837--853.

\bibitem{Sharma} Sharma, S.; Chandok, S. Split fixed point problems for quasi-nonexpansive mappings in Hilbert spaces. {\em Politehn. Univ. Bucharest Sci. Bull. Ser. A Appl. Math. Phys.} {\bf 86} (2024), no. 1, 109--118.

\bibitem{Shi} Shi, Y.; Zhu, Z.; Zhang, Y.. New algorithm for non-monotone and non-Lipschitz equilibrium problem over fixed point set of a quasi-nonexpansive non-self mapping in Hilbert space. {\em Appl. Anal. Optim.} {\bf 7} (2023), no. 3, 279--290.

\bibitem{Tiam} Tiammee, S.; Tiammee, J. Weak convergence of inertial proximal point algorithm for a family of nonexpansive mappings in Hilbert spaces.  {\em Carpathian J. Math.} {\bf 40} (2024), no. 1, 173--185.

\bibitem{Uba} Uba, M. O.; Onyido, M. A.; Udeani, C. I.; Nwokoro, P. U. A hybrid scheme for fixed points of a countable family of generalized nonexpansive-type maps and finite families of variational inequality and equilibrium problems, with applications. {\em Carpathian J. Math.} {\bf 39} (2023), no. 1, 281--292.


\bibitem{Wang-Xu} Wang, Y. Q.; Xu, T. T.; Yao, Y.-C.; Jiang, B. N. Self-Adaptive Method and Inertial Modification for Solving the Split Feasibility Problem and Fixed-Point Problem of Quasi-Nonexpansive Mapping. {\em Mathematics} {\bf 2022},10,1612. https:// doi.org/10.3390/math10091612.

%\bibitem{Wang17b} Wang, Y. Q.; Fang, X. L. Viscosity approximation methods for the multiple-set split equality common fixed-point problems of demicontractive mappings. {\em J. Nonlinear Sci. Appl.} {\bf 10} (2017), no. 8, 4254--4268.

%\bibitem{Wang21} Wang, Y. Q.; Fang, X. L.; Guan, J.-L.; Kim, T.-H. On split null point and common fixed point problems for multivalued demicontractive mappings. {\em Optimization} {\bf 70} (2021), no. 5-6, 1121--1140.

%\bibitem{Wang17} Wang, Y. Q.; Kim, T.-H.;  Fang, X. L.; He, H. M. The split common fixed-point problem for demicontractive mappings and quasi-nonexpansive mappings. {\em J. Nonlinear Sci. Appl.} {\bf 10} (2017), no. 6, 2976--2985.

%\bibitem{Wang17} Wang, Y. Q.; Kim, T.-H.; Chen, R. D.; Fang, X. L. The multiple-set split equality common fixed point problems for demicontractive mappings without prior knowledge of operator norms. {\em J. Nonlinear Convex Anal.} {\bf 18} (2017), no. 10, 1849--1865.

%\bibitem{Wang17a} Wang, Y. Q.; Kim, T.-H.;  Fang, X. L.  Weak and strong convergence theorems for the multiple-set split equality common fixed-point problems of demicontractive mappings. {\em J. Funct. Spaces} {\bf 2017}, Art. ID 5306802, 11 pp.

%\bibitem{Wang18} Wang, Y. Q.; Liu, W.; Song, Y. L.; Fang, X. L. Mixed iterative algorithms for the multiple-set split equality common fixed-point problem of demicontractive mappings. {\em J. Nonlinear Convex Anal.} {\bf 19} (2018), no. 11, 1921--1932.

\bibitem{Xu11} Xu, H-K. Averaged mappings and the gradient-projection algorithm. {\em J. Optim. Theory Appl.} {\bf 150} (2011), no. 2, 360--378.

%\bibitem{Yang} Yang, L.; Zhao, F.; Kim, J. K. The split common fixed point problem for demicontractive mappings in Banach spaces. {\em J. Comput. Anal. Appl.} {\bf 22} (2017), no. 5, 858--863.


%\bibitem{Yao18} Yao, Y. H.; Liou, Y.-C.; Postolache, M. Self-adaptive algorithms for the split problem of the demicontractive operators. {\em Optimization} {\bf 67} (2018), no. 9, 1309--1319.

%\bibitem{Yao20} Yao, Y. H.; Yao, J.-C.; Liou, Y.-C.; Postolache, M. Iterative algorithms for split common fixed points of demicontractive operators without priori knowledge of operator norms. {\em Carpathian J. Math.} {\bf 34} (2018), no. 3, 459--466.

%\bibitem{Yu} Yu, Y. R.; Sheng, D. L. On the strong convergence of an algorithm about firmly pseudo-demicontractive mappings for the split common fixed-point problem. {\em J. Appl. Math.} {\bf 2012}, Art. ID 256930, 9 pp.

%\bibitem{Zhang} Zhang, C. J.; Li, Y.; Wang, Y. H. On solving split generalized equilibrium problems with trifunctions and fixed point problems of demicontractive multi-valued mappings. {\em J. Nonlinear Convex Anal.} {\bf 21} (2020), no. 9, 2027--2042.

%\bibitem{Zheng} Zheng, X. X.; Yao, Y.H.; Liou, Y.-C.; Leng, L. M. Fixed point algorithms for the split problem of demicontractive operators. {\em J. Nonlinear Sci. Appl.} {\bf 10} (2017), no. 3, 1263--1269.


%\bibitem {Zhou21} Zhou, Z.; Tan, B.; Li, S. X. An accelerated hybrid projection method with a self-adaptive step-size sequence for solving split common fixed point problems. {\em Math. Methods Appl. Sci.} {\bf 44} (2021), no. 8, 7294--7303.

%\bibitem {Zong} Zong, C. X.; Tang, Y. C. Iterative methods for solving the split common fixed point problem of demicontractive mappings in Hilbert spaces. {\em J. Nonlinear Sci. Appl.} {\bf 11} (2018), no. 8, 960--970.


\end{thebibliography}
\end{document}